\documentclass[letterpaper]{article}

\usepackage{amsmath,amsthm}
\usepackage{amsfonts}
\usepackage{graphicx}
\usepackage{multicol}
\usepackage{color}
\usepackage{amsthm}
\usepackage{enumerate}
\usepackage{amssymb}
\usepackage{mathrsfs}

\usepackage{tikz}

\newcommand{\bbold}{\mathbb}

\def\R { {\bbold R} }
\def\Q { {\bbold Q} }
\def\Z { {\bbold Z} }

\def\N { {\bbold N} }

\def \cM{\mathbf{M}}

\def \rk{\operatorname{rk}}
\def \trdeg{\operatorname{trdeg}}

\renewcommand\epsilon{\varepsilon}

\def \<{\langle}
\def \>{\rangle}
\def \ee{\preceq}

\def \((  {(\!(}
\def \)) {)\!)}

\def \k {{{\boldsymbol{k}}}}

\DeclareMathSymbol{\precequ}{\mathrel}{symbols}{"16}
\DeclareMathSymbol{\succequ}{\mathrel}{symbols}{"17}

\def \FL{\operatorname{FL}}
\def \DF{\operatorname{DF}}
\def \ACF{\operatorname{ACF}}
\def \ACF2{\operatorname{2ACF}}
\def \RCF{\operatorname{RCF}}
\def \RCF2{\operatorname{2RCF}}

\def \DCF{\operatorname{DCF}}

\newtheorem{theorem}{Theorem}[section]
\newtheorem{lemma}[theorem]{Lemma}
\newtheorem{prop}[theorem]{Proposition}
\newtheorem{cor}[theorem]{Corollary}

\theoremstyle{definition}

\newtheorem{definition}[theorem]{Definition}

\theoremstyle{remark}

\newtheorem*{remark}{Remark}

\let\oldi\i
\let\oldj\j
\renewcommand\i{\relax\ifmmode{\boldsymbol{i}}\else\oldi\fi}
\renewcommand\j{\relax\ifmmode{\boldsymbol{j}}\else\oldj\fi}

\renewcommand\leq{\leqslant}
\renewcommand\geq{\geqslant}
\renewcommand\preceq{\preccurlyeq}

\renewcommand\le{\leq}
\renewcommand\ge{\geq}

\DeclareMathAlphabet{\mathbf}{OML}{cmm}{b}{it}

\DeclareFontFamily{U}{fsy}{}
\DeclareFontShape{U}{fsy}{m}{n}{<->s*[.9]psyr}{}
\DeclareSymbolFont{der@m}{U}{fsy}{m}{n}
\DeclareMathSymbol{\der}{\mathord}{der@m}{182}

\DeclareSymbolFont{der@m}{U}{fsy}{m}{n}
\DeclareMathSymbol{\derdelta}{\mathord}{der@m}{100}

\DeclareSymbolFont{imag@m}{OT1}{cmr}{m}{ui}
\DeclareMathSymbol{\imag}{\mathord}{imag@m}{105}

\DeclareFontFamily{OMS}{smallo}{}
\DeclareFontShape{OMS}{smallo}{m}{n}{<->s*[.65]cmsy10}{}
\DeclareSymbolFont{smallo@m}{OMS}{smallo}{m}{n}
\DeclareMathSymbol{\smallo}{\mathord}{smallo@m}{79}

\DeclareFontFamily{OMS}{largerdot}{}
\DeclareFontShape{OMS}{largerdot}{m}{n}{<->s*[.8]cmsy10}{}
\DeclareSymbolFont{largerdot@m}{OMS}{largerdot}{m}{n}
\DeclareMathSymbol{\largerdot}{\mathord}{largerdot@m}{15}

\newcommand\MR{\operatorname{MR}}

\DeclareMathSymbol{\llambda}{\mathord}{der@m}{108}
\DeclareMathSymbol{\rrho}{\mathord}{der@m}{114}

\newcommand{\equationqed}[1]{\[\pushQED{\qed}#1 \qedhere\popQED\]\let\qed\relax}
\newcommand{\alignqed}[1]{\begin{align*}\pushQED{\qed} #1 \qedhere\popQED\end{align*}\let\qed\relax}

\title{\normalsize\textbf{BOUNDED PREGEOMETRIES AND PAIRS OF FIELDS}}
\author{\scriptsize LEONARDO \'ANGEL\footnote{This paper was written during a visit of the first author in the fall of 2016 to the University of Illinois at Urbana-Champaign.}\quad and LOU VAN DEN DRIES }
\date{}

\begin{document}

\maketitle

\qquad \qquad \qquad{\it For Francisco Miraglia, on his 70th birthday}

\begin{abstract} \noindent
A definable set
in a pair $(K,\k)$ of algebraically closed fields is co-analyzable
relative to the subfield $\k$ of the pair
if and only if it is almost internal to $\k$. To prove this and some related 
results for tame pairs of real closed fields we introduce a
certain kind of ``bounded'' pregeometry for such pairs.   

\end{abstract}

\section{Introduction}

\medskip\noindent
The dimension of an algebraic variety equals the transcendence degree of its function field over the field of constants. 
We can use {\em transcendence degree\/} in a similar way to assign
to definable sets in algebraically closed fields a dimension with good properties such as definable dependence on parameters. Section 1 contains a general setting for model-theoretic pregeometries and proves basic facts about the dimension function associated
to such a pregeometry, including definable dependence on parameters if the pregeometry is {\em bounded}.

One motivation for this came from the following issue concerning  pairs $(K,\k)$ where $K$ is an algebraically closed field and $\k$ is a proper algebraically closed subfield; we refer to such a pair $(K,\k)$ as a {\em pair of algebraically closed fields},
and we consider it in the usual way as an $L_U$-structure, where  $L_U$ is the language $L$ of rings augmented by a unary
relation symbol $U$. Let 
$(K,\k)$ be a pair of algebraically closed fields, and let
$S\subseteq K^n$ be definable in $(K,\k)$. Then we have several notions of $S$ being controlled by $\k$, namely:
{\em $S$ is internal to $\k$}, {\em $S$ is almost internal to $\k$}, and
{\em $S$ is co-analyzable relative to $\k$}. The first notion is a-priori
more narrow than the second, which in turn is more narrow than
the third one. One application of our
results is that in the case at hand, {\em almost internal\/} coincides with {\em co-analyzable relative to $\k$}, both
coinciding with having finite Morley rank. However,
{\em internal to $\k$} is strictly stronger here.
(At the end of this introduction we provide definitions of {\em internal\/} and the like.)

\medskip\noindent
More precisely, for $(K,\k)$ as above we introduce a certain pregeometry on any sufficiently saturated elementary extension
and use this to associate to any nonempty set $S\subseteq K^n$ a dimension 
$\dim_2 S\in \N$ such that $\dim_2 K^n=n$ and $\dim_2 \k^n=0$.
We also assign to the empty subset of $K^n$ the dimension
$\dim_2 \emptyset:= -\infty$. For sets $S\subseteq K$ that are definable in $(K,\k)$ we now 
have the following dichotomy, reminiscent of strong minimality:

\begin{prop}\label{p} Either $\dim_2 S \le 0$, or $\dim_2 (K\setminus S) \le 0$.
\end{prop}

\noindent
For nonempty $S\subseteq K^n$ definable in $(K,\k)$ we have: 

\begin{theorem}\label{mt} The following are equivalent:
\begin{itemize}
\item $\dim_2 S=0$;
\item $S$ is almost internal to $\k$;
\item $S$ is co-analyzable relative to $\k$;
\item $S$ has finite Morley rank.
\end{itemize}
\end{theorem}

\noindent
We also consider pairs $(K,\R)$ where $K$ is a proper real closed
field extension of the real field $\R$. These pairs are construed
in the usual way as $L_{U}$-structures, and for such a pair $(K,\R)$ and $S\subseteq K$ definable in $(K,\k)$ we have by Dolich, Miller and Steinhorn~\cite[7.4]{Dol} a dichotomy similar to that of Proposition~\ref{p}:

\begin{prop}\label{pr} Either $S$ is discrete, or $S$ has nonempty interior. 
\end{prop}

\noindent
Here $K$ is equipped with its order topology. For a pair $(K,\R)$ as above we also have a similar dimension $\dim_2 S\in \N\cup\{-\infty\}$ for $S\subseteq K^n$, and for nonempty $S\subseteq K^n$ definable in $(K,\R)$ we have an analogue of Theorem~\ref{mt}:

\begin{theorem}\label{mtr} The following are equivalent:
\begin{itemize}
\item $\dim_2 S=0$;
\item $S$ is internal to $\R$;
\item $S$ is co-analyzable relative to $\R$;
\item $S$ is discrete.
\end{itemize}
\end{theorem}

\noindent
The key equivalence here is ``$\dim_2 S = 0\Leftrightarrow S \text{ is discrete}$'' for $n=1$,
and this comes from the paper \cite{Dol} that we already mentioned. This equivalence is needed to show that the relevant
pregeometry that underlies our dimension function $\dim_2$ is
bounded, which in turn is needed to obtain some of the other 
equivalences and the definable dependence of $\dim_2$ on parameters. Our interest in these equivalences came from \cite{D3}, where certain expansions of such tame pairs were shown to have similar properties, but where {\em co-analyzable\/}
turned out to be more general than {\em internal}. 

As we were finishing this paper, Matthias Aschenbrenner informed us of
Fornasiero's article \cite{F}, which has indeed considerable overlap with
ours.  
The pregeometry we
associate here to pairs of algebraically closed fields and the corresponding
dimension function does occur there, but not the above results on {\em tame\/} pairs of real closed fields. On the other hand, \cite{F} covers
{\em dense\/} pairs of real closed fields.

\subsection*{Notations and Preliminaries} We let $m,n$, sometimes subscripted, range over $\N=\{0,1,2,\dots\}$, and we let $\kappa$ range over infinite cardinals. For a relation $R\subseteq P\times Q$
between sets $P$ and $Q$ and $p\in P$, $E\subseteq P$ we set 
$$ R(p)\ :=\ \{q\in Q:\ (p,q)\in R\},\qquad R(E):=\bigcup_{p\in E}R(p).$$
Throughout $L$ is a one-sorted language, but in
some places we specify $L$ further. To define ``internal'' and the like,
let $\cM=(M;\cdots)$ be an $L$-structure, and let $D\subseteq M$ be $0$-definable with $|D|\ge 2$,
where here and below ``definability'' refers to definability in the 
ambient structure
$\cM$. Let also $S\subseteq M^n$ be definable. 

Then
$S$ is said to be {\em internal to $D$\/} if $S=f(E)$ for some definable
$E\subseteq D^m$ and definable map $f: E \to M^n$. Note that if 
$S$ is internal to $D$, then so is
every definable subset of $S$, and so is $g(S)$ for any definable map $g: S \to M^{r},\ r\in \N$. The union of two definable subsets of $M^n$ that are internal to $D$ is internal
to $D$. If $S$ is internal to $D$ and $S'\subseteq M^{n'}$ is definable and internal to $D$, then $S\times S'\subseteq M^{n+n'}$ is internal to $D$. 

The set $S$ is said to be {\em almost internal to $D$\/} if there is a definable relation
$R\subseteq E\times S$ with definable $E\subseteq D^m$
and a $d\in \N$ such that $R(E)=S$ and $|R(e)|\le d$ for all
$e\in E$. The closure properties of internality listed above also hold for almost internality. Clearly: 

\begin{lemma} \label{dto} If $S$ is almost internal to $D$ and there exists a definable total ordering on $M$, then $S$ is internal to $D$.
\end{lemma}


\medskip\noindent
To define co-analyzability relative to $D$, we assume for simplicity that $\mathbf M$ is $\omega$-saturated; \cite[Section 6]{D3} shows how to drop this assumption. 
By recursion on $r\in \N$ we 
define {\em $S$ is co-analyzable in
$r$ steps} (tacitly: relative to $\mathbf M$ and $D$):
\begin{enumerate}
\item[(C$_{0}$)] $S$ is co-analyzable in $0$ steps iff $S$ is finite;
\item[(C$_{r+1}$)] $S$ is co-analyzable in $r+1$ steps iff for some definable set 
$R\subseteq D\times M^n$,
\begin{enumerate}
\item the natural projection $D\times M^n \to M^n$ maps $R$ onto
$S$;
\item for each $c\in D$, the section $R(c):=\big\{s\in M^n:(c,s)\in R\big\}$ above~$c$ is
co-analyzable  in $r$ steps. \end{enumerate}
\end{enumerate}
We call $S$ {\bf co-analyzable} if $S$ is co-analyzable  in $r$ steps for some~$r$. The closure properties of internality listed above also hold for co-analyzability. 

It is easy to check for $S$ as above that
$$ \text{ internal to }D\ \Longrightarrow\ \text{almost internal to }D\ \Longrightarrow\ \text{co-analyzable relative to }D.$$
Co-analyzability is a robust property. This is also clear from
the following result, \cite[Proposition 2.4]{HHM}, where we assume $L$ is countable:

\begin{prop}\label{hmm} 
Let $T$ be a complete $L$-theory and $D(x)$ an $L$-formula that defines in each model of $T$ a set with more than one element. 
Then the following conditions on an $L$-formula $\varphi(y)$ with $y=(y_1,\dots,y_n)$ are equivalent:
\begin{enumerate}
\item[\textup{(i)}] for some  $\mathbf M\models T$, $\varphi(M^n)$ is co-analyzable relative to $D(M)$;
\item[\textup{(ii)}] for all $\mathbf M\models T$, $\varphi(M^n)$ is co-analyzable relative to $D(M)$;
\item[\textup{(iii)}] for all  $\mathbf M\models T$, if 
$D(M)$ is countable, then so is $\varphi(M^n)$;
\item[\textup{(iv)}] for all $\mathbf M\preceq\mathbf M^*\models T$, if
$D(M)=D(M^*)$, then $\varphi(M^n)=\varphi((M^*)^n)$.
\end{enumerate}
\end{prop}

\section{Pregeometries, Dimension, and Boundedness}\label{pregdim}

\noindent
Let $x_1,x_2, x_3,\dots, y, y_1, y_2, y_2,\dots$ be distinct
variables and let
$$\Phi\ :=\ \big(\phi_i(x_1,\dots, x_{m_i},y)\big)_{i\in I}$$ be 
a family of $L$-formulas.
Let $\cM=(M;\cdots)$ be an $L$-structure. We associate to each parameter set 
$A\subseteq M$ its $\Phi$-image 
$$\Phi(A)\ :=\ \{b\in M:\ \cM\models \phi_i(a,b) \text{ for some $i\in I$ 
and } a\in A^{m_i}\}.$$
By a $\Phi$-set in $\cM$ we mean a subset of $M$ which for some $i\in I$
has the form 
$$\phi_i(a_1,\dots, a_{m_i},M):= \{b\in M:\ \cM\models \phi_i(a_1,\dots, a_{m_i},b)\} \qquad( a_1,\dots, a_{m_i}\in M).$$
Next, let $\Sigma$ be a set of $L$-sentences. We say that 
$(\Phi,\Sigma)$ {\em defines a pregeometry\/} if for every model $\cM=(M;\cdots)$ of
$\Sigma$, the operation $A\mapsto \Phi(A)$ on the power set of $M$
is a pregeometry on $M$, to be called the $(\Phi,\Sigma)$-pregeometry of $\cM$.

\medskip\noindent
{\bf Examples.}

\medskip\noindent
1. Let $F$ be a field and let $L:= L_F$ be the usual language of vector 
spaces over $F$ (with a unary function symbol for each scalar). For 
$c=(c_1,\dots, c_m)\in F^m$, let $\phi_c(x_1,\dots, x_m,y)$
be the formula $c_1x_1+ \cdots + c_mx_m=y$, and let $\Phi_F$ be the family
$\big(\phi_c(x_1,\dots, x_m,y)\big)_{m\in \N, c\in F^m}$. Then for each vector space
$V$ over $F$ the $\Phi_F$-sets in $V$ are the one-element subsets of $V$, and
for $A\subseteq V$,
$$\Phi_F(A)\ =\ \text{the $F$-linear span of $A$}.$$ 
Thus $(\Phi_F, \Sigma_F)$ defines a pregeometry, where $\Sigma_F$ is the usual
set of axioms for vector spaces over $F$.

\medskip\noindent
2. Let $L$ be the language $\{0,1,-,+,\cdot\}$ of rings. To a polynomial
$p(X_1,\dots, X_m,Y)$ in $\Z[X_1,\dots, X_m,Y]$ we associate an $L$-formula
$$  p(x_1,\dots, x_m,y)=0\wedge  p(x_1,\dots, x_m,Y)\ne 0$$
to be denoted by $\phi_p(x_1,\dots, x_m,y)$. (Here the conjunct
$p(x_1,\dots, x_m,Y)\ne 0$ stands for an $L$-formula $\psi_p(x_1,\dots, x_m)$
such that for all
fields $K$ and $a\in K^m$, $K\models \psi_p(a)$ iff
$p(a,Y)\ne 0$ in $K[Y]$.) Let $\Phi_{\FL}$ be the family
$\big(\phi_p(x_1,\dots, x_m,y)\big)$ indexed by the above $(m,p)$. 
The $\Phi_{\FL}$-sets in a field $K$ are its finite subsets
(the zerosets of polynomials
$P(Y)\in K[Y]^{\ne}$), and for
$A\subseteq K$,
$$\Phi_{\FL}(A)\ =\ \{b\in K:\ b \text{ is algebraic over the subfield of $K$ generated by $A$}\}.$$
Thus $(\Phi_{\FL}, \text{FL})$ defines a pregeometry.

\bigskip\noindent
Returning to the general setting, assume that $(\Phi, \Sigma)$ defines a 
pregeometry. Let $\cM=(M;\cdots)\models \Sigma$. For $A\subseteq M$ and $n\ge 1$, a routine induction on $n$ yields a set
$I_n=I_n(\vec x,y_1,\dots, y_n)$ of $L$-formulas $\psi(x_1,x_2,\dots; y_1,\dots, y_n)$ with the following property: for all $b_1,\dots, b_n\in M$,
\begin{align*} b_1,\dots, b_n \text{ are  $(\Phi, \Sigma)$-independent over $A$}\ \Longleftrightarrow\  &(b_1,\dots, b_n) \text{ realizes the partial}\\
 & \qquad \text{type } I_n(A,y_1,\dots, y_n),
\end{align*}
where $I_n(A,y_1,\dots, y_n)$ consists of the formulas $\psi(a_1,a_2\dots;y_1,\dots, y_n)$ such that $\psi(x_1, x_2,\dots; y_1,\dots, y_n)\in I_n$ and $a_1, a_2,\dots \in A$. Moreover, $I_n$ can be 
taken to depend only on $(L,\Phi,n)$, not on $(\Sigma, \cM, A)$. 

\medskip\noindent
For $A\subseteq B \subseteq M$ we have the cardinal 
\begin{align*} \text{rk}(B|A)\ :=\ & \text{the size of any basis of 
$\Phi(B)$ over $\Phi(A)$}\\
&\text{with respect to the $(\Phi, \Sigma)$-pregeometry on $\cM$.}
\end{align*}
Let now $\cM=(M;\cdots)\models \Sigma$. We use the above rank to assign to
any set $S\subseteq M^n$ a dimension $\dim S\in \{-\infty,0,1,\dots, n\}$
as follows. First, suppose 
$S\ne \emptyset$. Then we
take a $\kappa>|M|$ and a $\kappa$-saturated elementary 
extension 
$(\cM^*, S^*)=(M^*;\dots, S^*)$ of the $L_n$-structure $(\cM,S)$, where
$L_n$ is $L$ augmented by an $n$-ary relation symbol as a name for $S$. We set
$$\dim S\ :=\ \max\{\text{rk}(Ms|M):\ s\in S^*\}$$
where the rank is with respect to the $(\Phi, \Sigma)$-pregeometry of $\cM^*$.
Also we put $\dim S :=-\infty$ if $S=\emptyset$. 

\begin{remark}
By the previous characterization of 
independence over a parameter set in terms of realizing a partial type,
$\dim S$ does not depend on our choice of the (suitably saturated)
elementary extension $(\cM^*, S^*)$ of $(\cM,S)$.
\end{remark}
\bigskip\noindent
Here are some easy consequences of this definition of {\em dimension}:

\begin{lemma}\label{dim2} Let $S, S_1, S_2\subseteq M^n$. Then: \begin{enumerate}
\item[ \textup{(i)}] if $S$ is finite and nonempty, then $\dim S =0$;
\item[ \textup{(ii)}] $\dim(S_1\cup S_2)=\max(\dim S_1, \dim S_2)$;
\item[ \textup{(iii)}] If $S_1 \subseteq S_2$, then $\dim(S_1)\leq \dim(S_2)$;
\item[ \textup{(iv)}] $\dim S^{\sigma}=\dim S$ for each permutation $\sigma$ of $\{1,\dots,n\}$, where
$$S^\sigma\ :=\ \big\{\big(y_{\sigma(1)},\dots,y_{\sigma(n)}\big):\ (y_1,\dots,y_n)\in S\big\};$$
\item[ \textup{(v)}] if $m\le n$ and $\pi\colon M^n \to M^m$ is given by $\pi(y_1,\dots, y_n)=(y_1,\dots, y_m)$, then $\dim \pi(S) \le \dim S$;
\item[ \textup{(vi)}] if $\dim S=m$, then $\dim \pi(S^{\sigma})=m$ for some
$\sigma$ as in \textup{(iv)} and $\pi$ as in \textup{(v)}. 
\end{enumerate}
\end{lemma}

\bigskip\noindent
Note also that each nonempty $\Phi$-set has dimension zero. In fact:  

\begin{lemma}\label{dim3} For nonempty $S\subseteq M$, 
$$\dim S=0\ \Longleftrightarrow\  S \text{ is contained in a finite union of $\Phi$-sets}.$$
\end{lemma}
\begin{proof}
Suppose $\dim S=0$. Then $\text{rk}(Ms|M)=0$ for all $s\in S^*$, so for every $s\in S^*$ we have $i(s)\in I$ and $a_s \in M^{m_i(s)}$ such that $\cM^*\models \phi_{i(s)}(a_s,s)$. Hence
$$S^*\ \subseteq\ \bigcup\limits_{i(s)\in I}\phi_{i(s)}(a_s,M^*).$$
Then saturation of $(\cM^*,S^*)$ yields a finite $J\subseteq I$ and elements
$a_i\in M^{m_i}$ for $i\in J$ such that 
$$S\ \subseteq\ \bigcup\limits_{i\in J}\phi_{i}(a_i,M).$$
For $\Leftarrow$, use property \textup{(ii)} of lemma \ref{dim2}.
\end{proof}

\noindent
For $S\subseteq M^n$, $\dim_{\cM} S$ stands for $\dim S$ 
in case we need to indicate $\cM$. 

\begin{lemma}\label{dim5} Suppose that $S\subseteq M^n$ and $(\cM,S)\ee (\cM', S')$ as $L_n$-structures. Then 
$\dim_{\cM} S =  \dim_{\cM'} S'$.
\end{lemma}

\begin{proof}
Let $(\cM^*,S^*)$ be an $\kappa$-saturated elementary extension of $(\cM', S')$ with $\kappa>|M'|$. We can assume that $S\ne \emptyset$. Using $(\cM^*, S^*)$ it follows from the definition of $\dim$ that $\dim_{\cM} S\geq \dim_{\cM'} S'$.  Let $\dim_{\cM}S=d$, and
take $s\in S^*$ with $\rk(Ms|M)=d$, and let $1\le i_1< \dots < i_d\le n$ be such that $s_{i_1},\dots, s_{i_d}$ are $(\Phi, \Sigma)$-independent over $M$. Then $(s_1,\dots, s_n)$ realizes the partial type $I_d(M,y_{i_1},\dots,y_{i_d})\cup\{R(y_1,\dots, y_n)\}$ in
$(\cM^*, S^*)$, where $R$ is the $n$-ary relation symbol for $S$. In view of $(\cM,S)\ee (\cM', S')$ it follows that the partial type $I_d(M',y_{i_1},\dots,y_{i_d})\cup\{R(y_1,\dots, y_n)\}$ can be realized in
$(\cM^*, S^*)$, and thus $\dim_{\cM'}S'\ge d$.  
\end{proof}

\noindent
This leads to the following:

\begin{lemma}\label{dim4} $\dim(S_1\times S_2)=\dim S_1 + \dim S_2$ for $S_1\subseteq M^m$ and $S_2\subseteq M^n$. 
\end{lemma}

\begin{proof} We may assume that $S_1$ and $S_2$ are nonempty.
The definition of $\dim$ yields $\dim S_1\times S_2\leq \dim S_1 + \dim S_2$. Let $\dim S_1=d_1$ and $\dim S_2=d_2$. Take an elementary extension
$(\cM', S_1', S_2')$ of $(\cM, S_1, S_2)$ with a point
$s_1\in S_1'$ such that $\rk(Ms|M)=d_1$. Next, by Lemma~\ref{dim5},
we can take an elementary extension $(\cM^*, S_1^*, S_2^*)$
of $(\cM', S_1', S_2')$ and a point $s_2\in S_2^*$ such that $\rk(M's_2|M')=d_2$. Hence $\rk(Ms_1s_2|M)\ge d_1+d_2$, and
thus $\dim S_1\times S_2\ge d_1+d_2$.   
\end{proof}

\noindent
Thus for nonempty $S\subseteq M^n$ we have: $\dim S = 0$ iff $\dim \pi_j(S)=0$
for $j=1,\dots,n$, where $\pi_j: M^n \to M$ is given by 
$\pi(y_1,\dots, y_n)=y_j$. 

So far we didn't assume that the sets 
$S\subseteq M^n$ are definable. Below we make further 
assumptions on our pregeometry and consider only definable sets.

\subsection*{Pregeometries that are nontrivial and bounded}
We call $(\Phi, \Sigma)$ 
{\it trivial\/} on the model $\cM=(M;\cdots)$ of $\Sigma$
if $M$ is a finite union
of $\Phi$-sets in $\cM$ (in which case $\dim S=0$ for every
nonempty set $S\subseteq M^n$). Note that $(\Phi, \Sigma)$
is trivial on every finite model of $\Sigma$. In order to focus on cases of interest we say that $(\Phi, \Sigma)$ is 
{\it nontrivial\/} if for every $\cM\models \Sigma$ it is not trivial on $\cM $ (so $\Sigma$ has no finite models).

Note that $(\Phi_F,\Sigma_F)$ and 
$(\Sigma_{\FL},\FL)$ in the examples 1 and 2 do not satisfy this nontriviality condition. This is easy to repair: let $\Sigma_F^{\infty}$ be a set of 
$L_F$-axioms whose models are the {\em infinite\/} vector spaces, and let
$\FL^{\infty}$  be a set of 
axioms in the language of rings whose models are the {\em infinite\/} fields.
Then $(\Phi_F, \Sigma_F^{\infty})$ and $(\Phi_{\FL}, \FL^{\infty})$ are nontrivial.

\bigskip\noindent
Note that if $(\Phi, \Sigma)$ is nontrivial, $\cM=(M;\cdots)\models \Sigma$ is 
$\kappa$-saturated, and $A\subseteq M$ is a parameter set
of size $< \kappa$, then $\Phi(A)\ne M$. This in turn yields:

\begin{cor} Suppose $(\Phi, \Sigma)$ is nontrivial.
Then for every $\cM\models\Sigma$ with
$\cM=(M;\dots)$ we have 
$\dim M = 1$, and thus $\dim M^n=n$ by Lemma~\ref{dim4}. 
\end{cor}   

\noindent
Next we consider definable sets in models of
$\Sigma$. It is clearly desirable 
that the dimension of such a set be invariant under definable bijections
and varies definably in definable families. To obtain this and more
we introduce a stronger condition: We say that $(\Phi, \Sigma)$ is
{\em bounded\/} if it is nontrivial and for every model $\cM=(M;\cdots)$ of $\Sigma$ and definable
$S\subseteq M^{n+1}$ the set $\{a\in M^n:\ \dim S(a)=0\}$ is definable; here $S(a):=\{b\in M: (a,b)\in S\}$. 

\begin{lemma}\label{dim8} Assume $(\Phi, \Sigma)$ is nontrivial.
Then $(\Phi, \Sigma)$ is bounded if and only if for every model 
$\cM=(M;\cdots)$ of $\Sigma$ and definable
$S\subseteq M^{n+1}$ there are $($not necessarily distinct$)$ indices $i(1),\dots, i(r)$ in $I$ with $r\in \N$, 
such that for all $a\in M^n$, if $\dim S(a)=0$, then
$$S(a)\ \subseteq\ \phi_{i(1)}(a_1,M)\cup \cdots \cup  \phi_{i(r)}(a_r,M)$$
for some tuples $a_1\in M^{m_{i(1)}},\dots, a_r\in M^{m_{i(r)}}$. 
\end{lemma}

\begin{proof}
Suppose that $(\Phi, \Sigma)$ is bounded, $\cM=(M;\cdots)\models \Sigma$, and $S\subseteq M^{n+1}$ is definable. 
Let $(\cM^*,S^*)$ be an $|M|^+$-saturated elementary extension of 
$(\cM,S)$. For $a\in M^n$ we have $\dim_{\cM} S(a)= \dim_{\cM^*} S^*(a)$ by Lemma~\ref{dim5}. The set 
$$ P\ :=\ \{a\in (M^*)^n:\ \dim_{\cM^*}S^*(a)=0\}$$ is definable in $\cM^*$. For a tuple $\i=(i(1),\dots,i(r))\in I^r$ with $r\in \N^{\ge 1}$, let $P_{\i}$ be the set of all 
$a\in (M^*)^n$ with $S^*(a)\ne \emptyset$ for which there exist tuples $a_1\in (M^*)^{m_{i(1)}},\dots, a_r\in (M^*)^{m_{i(r)}}$ such that 
$$S^*(a)\ \subseteq\ \phi_{i(1)}(a_1,M^*)\cup \cdots \cup  \phi_{i(r)}(a_r,M^*).$$
Note that each set $P_{\i}$ is definable over $M$ in $\cM^*$, and that $P=\bigcup_{\i}P_\i$ by Lemma~\ref{dim3}. By saturation
this yields a finite set $F$ of tuples $\i$ as above such that
$P=\bigcup_{\i\in F} P_{\i}$. By applying this to tuples 
$a\in M^n$ this easily yields the ``only if'' direction. The converse is
clear. 
\end{proof}

\noindent
Arguing as in section 1 of \cite{D1} we get the following:

\begin{prop}\label{pro1} Suppose $(\Phi, \Sigma)$ is bounded. Let $\cM=(M;\cdots)$ be a model of $\Sigma$, and let $S\subseteq M^m$ and $f: S \to M^n$ be
 definable.  Then \begin{enumerate}
\item[\textup{(1)}] $\dim S \ge \dim f(S)$, in particular, $\dim S = \dim f(S)$ if $f$ is injective;
\item[\textup{(2)}] for $0\le i\le m$ the set $B_i:=\{b\in M^n:\ \dim f^{-1}(b)=i\}$ is 
definable, and $\dim f^{-1}(B_i) = i+ \dim B_i$.
\end{enumerate}
\end{prop}

\begin{cor}\label{corpro1} Suppose $(\Phi, \Sigma)$ is bounded. Let $\cM=(M;\cdots)$ be a model of $\Sigma$, let $D\subseteq M$ be $0$-definable with $\dim D=0$, and let the nonempty set
$S\subseteq M^m$ be definable and
co-analyzable relative to $D$. Then $\dim S=0$.
\end{cor} 
\begin{proof} By passing to a suitable elementary extension of 
$\cM$ we can arrange that $\cM$ is $\omega$-saturated. Then
an easy induction on $r\in \N$ using Proposition~\ref{pro1} shows that if $S$ is co-analyzable in $r$ steps relative to $D$, then
$\dim S =0$.  
\end{proof}

\medskip\noindent
{\bf Examples.}

\medskip\noindent
3. Let $L$ be the language of rings, let $\text{ACF}$ be the usual set of axioms (in $L$) for algebraically closed fields, and let $\Phi_{\FL}$ be as in 
Example 2. Then $(\Phi_{\FL},\operatorname{ACF})$ defines a pregeometry, and $(\Phi_{\FL},\operatorname{ACF})$ is bounded: more precisely, by \cite{D1} there are for 
any $L$-formula $\phi(x_1,\dots, x_n,y)$ polynomials $P_1,\dots,P_r\in \Z[X_1,\dots,X_n,Y]$ such that for all
$K\models \operatorname{ACF}$ and $a\in K^n$, if $\phi(a,K)$ is finite, then $$\phi(a,K)\ \subseteq\ \{b\in K:\ P_i(a,b)=0\}$$ for some $i\in \{1,\dots,r\}$ with $P_i(a,Y)\neq 0$. The
dimension function associated to $(\Phi_{\FL}, \operatorname{ACF})$ is called
{\em algebraic dimension\/}, and is denoted by $\text{algdim}_K$ if we wish to indicate the dependence on $K$. Note that for $K\models  \operatorname{ACF}$ and
nonempty 
definable $S\subseteq K^n$ we have: $\text{algdim}_K S$ is the largest $m$ for which there exist polynomials $P_1,\dots,P_m \in K[X_1,\dots,X_n]$ that are algebraically independent on $S$; the latter means that for every $f\in K[Y_1,\dots,Y_m]^{\neq}$ there exists an $s\in S$ such that
 $f(P_1(s),\dots,P_m(s))\ne 0$.

\medskip\noindent
Among several variants of the above in \cite{D1}, we indicate one: let $L$ be the language of ordered rings (the language of rings augmented by $\le$), and let  $\text{RCF}$ be the usual set of axioms in this language for real closed ordered fields.
Then the above statements for $(\Phi_{\FL},\operatorname{ACF})$
go through for $(\Phi_{\FL},\operatorname{RCF})$.

\medskip\noindent
4. Let $L_{\der}$ be the language of rings augmented by the unary function symbol $\der$, and let $\Q\{Y_1,\dots,Y_n\}$ be the (differential) ring of differential polynomials in $Y_1,\dots,Y_n$
over $\Q$ (with the trivial derivation on $\Q$). To each differential polynomial $P(X_1,\dots, X_m,Y)\in \Q\{X_1,\dots,X_m,Y\}$ we associate an $L_{\der}$-formula
$$ P(x_1,\dots, x_m,y)=0\ \wedge\  P(x_1,\dots, x_m,Y)\ne 0$$
to be denoted by $\phi_P(x_1,\dots,x_m,y)$. 
Let 
$\Phi_{\DF}$ the family $\big(\phi_P(x_1,\dots,x_m,y)\big)$ indexed by the above
$(m,P)$. Let $\DF$ be the usual set of $L_{\der}$-sentences whose models are the differential fields of characteristic $0$ (with $\der$ interpreted as the distinguished derivation). Then $(\Phi_{\DF},\DF)$ defines a pregeometry. Let now $\DCF$ be a set of
$L_{\der}$-sentences whose models are the differentially closed fields. Then $(\Phi_{\DF},\DCF)$ is bounded, as pointed out in [1]. 


\section{Pairs of algebraically closed fields}\label{pairs}

\medskip\noindent
Let $L_U$ be the language of rings augmented by a unary relation symbol $U$. Let $\ACF2$ be a set of $L_U$-sentences whose models are the pairs $(K,\k)$ of algebraically closed fields; recall from the introduction that this includes the requirement that $\k$ is a {\em proper\/} algebraically closed subfield of $K$. 

\subsection*{A pregeometry for pairs of algebraically closed fields}
To any $P(X,Y,Z)\in \Z[X,Y,Z]$ with $X=(X_1,\dots,X_m)$, $Y=(Y_1,\dots,Y_n)$, and $Z$ a single indeterminate, we associate the $L_U$-formula $\phi_P(y,z)$ given by
$$ \exists x \big(U(x)\ \wedge\ P(x,y,z)=0\ \wedge\ P(x,y,Z)\neq 0\big),$$
where $\exists x$ abbreviates $\exists x_1\dots\exists x_m$ and
$U(x)$ stands for $U(x_1)\wedge \dots \wedge U(x_m)$.
Let $\Phi_{2}$ be the family $\big(\phi_P(y,z)\big)$ indexed by the above $(m,n,P)$. Let $(K,\k)$ be a pair of algebraically closed fields. For a set $A\subseteq K$ the $\Phi_{2}$-image of $A$, 
$$\Phi_2(A)\ :=\ \lbrace b\in K:\ (K,\k)\models \phi_P(a,b) \text{ for some } \phi_P  \text{ as above and some } a \in  A^n\rbrace$$
\medskip\noindent
is  the algebraic closure of $\k(A)$ in $K$. Thus $(\Phi_2,\ACF2)$ defines a pregeometry. For sets $A\subseteq B\subseteq K$ we have
$$\rk_2(B|A)\ =\ \trdeg(\k(B)|\k(A))$$
where $\text{rk}_2$ refers here to the rank associated to the pregeometry on the set $K$ defined by $(\Phi_2, \ACF2)$,
and $\trdeg(F|E)$ is the transcendence degree of a field $F$ over a subfield $E$. This allows us to assign to every $S\subseteq K^n$ a dimension $\dim_2 S\in \N\cup\{-\infty\}$
as described in Section 1:

\begin{definition} Let $S\subseteq K^n$ and let $(K^*,\k^*,S^*)$ be a $\kappa$-saturated elementary extension of $(K,\k,S)$, with $\kappa > |K|$. If $S\ne \emptyset$ we set
$$\dim_2 S\ :=\ \max\lbrace \text{rk}_2(Ks | K): s \in S^*\rbrace,$$
\medskip\noindent
where the rank is with respect to the $(\Phi_2,\ACF2)$-pregeometry on $K^*$, and if $S=\emptyset$, then $\dim_2 S:=-\infty$. In more familiar terms, if $S\ne \emptyset$, then
$$ \dim_2 S\ =\ \max\{\trdeg\big(K\k^*(s)|K\k^*\big):\ s\in S^*\}$$
where $K\k^*$ denotes the subfield of $K^*$ generated by $K$ and $\k^*$.
\end{definition}

\medskip\noindent
In particular, Lemmas \ref{dim2}, \ref{dim3}, \ref{dim5}, and~\ref{dim4}
hold for our dimension $\dim_2$. 

\medskip\noindent
For example, if $S\subseteq K$ and $\dim_2 S=0$, then $S$ is contained in a finite union of $\Phi_2$-sets by Lemma~\ref{dim3}, and so there are polynomials $P_1, \dots, P_r\in \Z[X,Y,Z]$ with $r\in \N$ and $X, Y, Z$ as above, and a tuple $a\in K^{n}$ such that for all $b\in S$,
$$P_i(u,a,b)=0, \qquad  P_i(u,a,Z)\neq 0$$ for some
$i\in \{1,\dots,r\}$ and some $u\in \k^{m}$.

\medskip\noindent
\begin{remark} The characterization of $\dim_2$ in terms of transcendence degree yields $\dim_2 \k=0$. A remark at the end of section 2 in \cite{D2} gives $\dim_2 K>0$. Thus $(\Phi_2, \ACF2)$
is nontrivial. In particular, $\dim_2 K^n =n$.
\end{remark}

\subsection*{Boundedness of the pregeometry}
To prove that $(\Phi_2, \ACF2)$ is bounded we use a quantifier reduction for $\ACF2$, as in Lemma~\ref{dim9} below. Let $y=(y_1,\dots, y_n)$ be a tuple of distinct variables.
A {\it special\/} $L_U$-formula in $y$ is an $L_U$-formula 
$$(*) \qquad \qquad \exists x \big(U(x) \wedge \phi(x,y)\big)\qquad \qquad$$ 
where $x=(x_1,\dots,x_m)$ and $\phi(x,y)$ is an $L$-formula. If $(K,\k)$ is a pair of algebraically closed fields, then $(K,\k)$ is a model of the $L_U$-theory of {\em algebraically closed fields with a small multiplicative set\/} as  defined in \cite{D2}. Thus, as a special case of \cite[Theorem 3.8]{D2} we have:

\begin{lemma}\label{dim9} Every $L_U$-formula $\psi(y)$ is 
$\ACF2$-equivalent to a boolean combination of special $L_U$-formulas in $y$.
\end{lemma}

\noindent
Any conjunction as well as any disjunction of special 
$L_U$-formulas in $y$ is clearly equivalent to a special $L_U$-formula in $y$. Using that ACF has QE, every special $L_U$-formula in $y$
is $\ACF2$-equivalent to a disjunction of special $L_U$-formulas
in $y$ as in $(*)$ where $\phi(x,y)$ is a conjunction 
$$(**)\qquad \qquad P_1(x,y)=\cdots=P_r(x,y)=0\ \wedge\ Q(x,y)\ne 0, \qquad \qquad$$
with $r\in \N^{\ge 1}$ and $P_1,\dots,P_r, Q\in \Z[X,Y]$. 
We say that an $L_U$-formula $\psi(y)$ is {\it very special\/} 
if it is a formula
as in $(*)$ with $\phi(x,y)$ as as in $(**)$.  
Using the remarks above and Lemma~\ref{dim9} we obtain:

\begin{lemma}\label{dim9+} Every $L_U$-formula
$\psi(y)$ is $\ACF2$-equivalent to a disjunction of formulas
$$\psi_0(y) \wedge \neg \psi_1(y) \wedge \cdots \wedge \neg \psi_r(y)$$ 
where $r\in \N^{\ge 1}$ and the $\psi_i(y)$ are very special $L_U$-formulas.
\end{lemma}

\noindent
Let $z$ be a new variable, and consider a very special $L_U$-formula 
$$\psi(y,z):  \quad\exists x\big(U(x) \wedge   P_1(x,y,z)=\cdots=P_r(x,y,z)=0 \wedge Q(x,y,z)\ne 0\big)$$
with $P_1,\dots, P_r, Q\in \Z[X,Y,Z],\ X=(X_1,\dots, X_m),\  Y=(Y_1,\dots, Y_n)$.

\begin{lemma}\label{dim9*} Let $(K,\k)$ be a pair of algebraically closed fields.
Then for all $a\in K^n$, either $\dim_2\psi(a,K)\le 0$, or
$\psi(a,K)$ is a cofinite subset of $K$. The set
$$\{a\in K^n:\  \dim_2\psi(a,K)\le 0\}$$
is defined in $(K,\k)$ by a formula $\psi^*(y)$ of the form
$\forall x\big(U(x)\to \phi(x,y)\big)$
where $\phi(x,y)$ is a quantifier-free $L$-formula.
\end{lemma}
\begin{proof} Let $a\in K^n$. Suppose that for some $u\in \k^m$
we have $$P_1(u,a,Z)=\cdots=P_r(u,a,Z)=0, \quad  Q(u,a,Z)\ne 0.$$
Then $\psi(a,K)$ is a cofinite subset of $K$ with $|K\setminus \psi(a,K)|\le \deg_Z Q(u,a,Z)$.

If there is no such $u$, then for every $u\in \k^m$, either
some $P_i(u,a,Z)\ne 0$ or $P_1(u,a,Z)=\cdots = P_r(u,a,Z)=Q(u,a,Z)=0$, and then
$\dim_2\psi(a,K)\le 0$. 

Thus $\{a\in K^n:\  \dim_2\psi(a,K)\le 0\}$ is defined in $(K,\k)$ by the formula
$\forall x\big(U(x) \to \phi(x,y)\big)$, where $\phi(x,y)$ is the formula
$$\bigvee_{i=1}^r \big(P_i(x,y,Z)\ne 0\ \vee\ P_1(x,y,Z)=\cdots = P_r(x,y,Z)=Q(x,y,Z)=0\big).$$ 
Thus $\psi^*(y)$ can be taken to depend only on $\psi(y,z)$, not on $(K,\k)$.
\end{proof}

\noindent
{\em Proof of Proposition~\ref{p}}. The
collection of sets $S\subseteq K$ that are definable in
$(K,\k)$ with $\dim_2 S \le 0$ or $\dim_2 (K\setminus S)\le 0$ is easily seen to be a boolean algebra of subsets of $K$. By Lemma~\ref{dim9*} this boolean algebra contains every set $\psi(a,K)$
where $\psi(y,z)$ is a very special $L_U$-formula and $a\in K^n$.
By Lemma~\ref{dim9+} this boolean algebra contains every set $S\subseteq K$ that is definable in $(K,\k)$. \qed

\begin{prop}\label{pro2} $(\Phi_2, \ACF2)$ is bounded.
\end{prop}
\begin{proof} Let $(K,\k)\models \ACF2$, and let $S\subseteq K^{n+1}$ be $0$-definable in $(K,\k)$; we need to show that
$\{a\in K^n:\ \dim_2 S(a)\le 0\}$ is definable in $(K,\k)$. If $S=S_1\cup \dots \cup  S_d$
with $0$-definable $S_1,\dots, S_d$, $d\in \N$, then
$\dim_2 S(a)\le 0$ iff $\dim_2 S_i(a)\le 0$ for $i=1,\dots,d$. 
Thus by Lemma~\ref{dim9+} we can reduce to the case that
$S$ is defined in $(K,\k)$ by a conjunction 
$$\psi_0(y,z) \wedge \neg \psi_1(y,z) \wedge \cdots \wedge \neg \psi_r(y,z)$$ 
where $r\in \N^{\ge 1}$ and the $\psi_i(y,z)$ are very special $L_U$-formulas. For $a\in K^n$,
$$S(a)\ =\ \psi_0(a,K)\cap \psi_1(a,K)^c\cap \cdots \cap \psi_r(a,K)^c,$$
so by Lemma~\ref{dim9*} we have: $\dim_2 S(a)\le 0$ iff
$\dim_2 \psi_0(a,K)\le 0$ or $\psi_j(a,K)$ is cofinite for some 
$j\in \{1,\dots,r\}$. The desired result now follows from the last part of Lemma~\ref{dim9*}.
\end{proof}



\section{Definable sets of dimension zero}\label{Sec3}

In this section we analyse in more detail definable sets of dimension zero (with respect to the dimension function $\dim_2$) in models $(K,\k)$ of $\ACF2$.

\subsection*{Almost internality and Morley rank in $\ACF2$}

Let $S\subseteq K^n$ be nonempty and definable in a pair $(K,\k)$ of
algebraically closed fields. Since $\dim_2 \k =0$, we obtain 
as a special case of Corollary~\ref{corpro1} that if $S$ is 
co-analyzable relative to $\k$, then $\dim_2 S =0$.
Here is a strong converse:

\begin{prop}\label{ai} If $\dim_2 S=0$, then $S$ is almost internal to $\k$.
\end{prop}
\begin{proof} We first consider the case $n=1$. By Lemma~\ref{dim3}
we can reduce to the case that $S$ is a nonempty $\Phi_2$-set. Then
we have a nonzero polynomial $P\in \Z[X,Y,Z]$ as in the beginning
of Section~\ref{pairs} and an $a\in K^n$ such that
$S=\Phi_P(a,K)$. Set $d:= \deg_Z P$. Then the relation
$R\subseteq \k^m \times K$ given by 
$$R(u,b)\ :\Longleftrightarrow\ P(u,a,b)=0 \text{ and }P(u,a,Z)\ne 0$$
is definable in $(K,\k)$, $|R(u)|\le d$ for all $u\in \k^m$, and
$R(\k^m)=S$. Thus $S$ is almost internal to $\k$.  

Next let $n>1$, and assume that $\dim_2 S=0$.
Then $\dim_2 \pi_i(S)=0$ for $i=1,\dots,n$. Using the result for $n=1$, each $\pi_i(S)$ is almost internal to $\k$, hence
their cartesian product $\pi_1(S)\times \cdots \times \pi_n(S)$
is almost internal to $\k$, and thus the subset $S$ of this
product set is as well. 
\end{proof}

\medskip\noindent
The theory $\ACF2$ (or rather each completion of it) is $\omega$-stable, and moreover, $\MR(\k)=1$ and $\MR(K)=\omega$; see
\cite{Bue} or \cite{D2}. We use this to show:

\begin{prop} Assume $S\ne \emptyset$. Then
$\dim_2 S=0\Longleftrightarrow \MR(S)$ is finite. 
\end{prop}
\begin{proof} From $\MR(\k)=1$ we obtain that if $S$ is almost internal to $\k$, then $\MR(S)$ is finite. In view of 
Proposition~\ref{ai} this gives the direction $\Rightarrow$.
For $\Leftarrow$ we can assume $n=1$ by using projections
as before. So $S\subseteq K$, and we assume that
$\MR(S)=0$. In view of $\MR(K)=\omega$, this excludes the possibility that $\dim_2(K\setminus S)\le 0$, so we must have $\dim_2 S=0$ by Proposition~\ref{p}.  
\end{proof}

\noindent
We now have established all equivalences of Theorem~\ref{mt}. 

\medskip\noindent
In the introduction we claimed that in this context, {\em internal to $\k$\/} is strictly stronger than {\em almost internal to $\k$}. We briefly sketch why. Take a pair $(K,\k)$ of algebraically closed fields of characteristic zero such that $K$ has infinite transcendence degree over $\k$,
let $t\in K\setminus \k$, and take $S:=\{b\in K: b^2\in t+ \k\}$.
Then $S$ is of course almost internal to $\k$, but $S$ is not
internal to $\k$. Suppose it were. Then
$S\subseteq \text{dcl}(\k(a))$ with $a\in K^n$, where 
$\text{dcl}(\k(a))$ is the definable closure of the field $\k(a)$ in $(K,\k)$. By considering automorphisms we see that $\text{dcl}(\k(a))=\k(a)$, so $S\subseteq \k(a)$, and so 
we have a strictly increasing sequence 
$$\k(\sqrt{t+1})\ \subset\ \k(\sqrt{t+1}, \sqrt{t+2})\ \subset\ \k(\sqrt{t+1}, \sqrt{t+2}, \sqrt{t+3})\ \subset\ \cdots$$
of subfields of $\k(a)$ containing $\k$, which is impossible.   

\section{Tame pairs of real closed fields}

\noindent
A {\bf tame pair of real closed fields\/} is a pair $(K, \k)$
of real closed fields with $\k$ a proper subfield of $K$ (and thus an ordered subfield
with respect to the unique field orderings of $K$ and $\k$) such that
$$\mathcal{O}\ =\ \k + \smallo \qquad (\text{tameness condition})$$
where $\mathcal{O}:=\{a\in K:\ |a|\le |b| \text{ for some }b\in \k \}$ is the 
convex hull of $\k$ in $K$ and $\smallo:=\{a\in K:\ |a| < b \text{ for every }b\in \k^{>}\}$ is the maximal ideal of the (convex) valuation ring 
$\mathcal{O}$ of $K$. Thus any real closed field $K$ containing $\R$
as a proper subfield yields a tame pair $(K,\R)$ of real closed fields, 
but $(\R, \k)$
with $\k$ the relative algebraic closure of $\Q$ in $\R$ is not a tame pair of
real closed fields, because the tameness condition is not satisfied.  

\medskip\noindent
Let $\RCF2$ be the $L_{U}$-theory of tame pairs $(K, \k)$ of real closed fields with $U$ interpreted as the underlying set of $\k$. We know from \cite{Macintyre68} that $\RCF2$ is complete. 
 
\medskip\noindent
Now $(\Phi_2,\RCF2)$ defines a pregeometry for the same reason as
$(\Phi_2, \ACF2)$ does. Moreover, $(\Phi_2, \RCF2)$ is nontrivial: take a real 
closed field extension $K$ of $\R$ such that $K$ has greater cardinality
than $\R$; every $\Phi_2$-set in $(K, \R)$ has clearly cardinality at most 
that of $\R$, so $K$ cannot be a finite union of $\Phi_2$-sets in $(K,\R)$;
it remains to appeal to the completeness of $\RCF2$. 

In the rest of this section we fix a tame pair $(K,\k)$ of real closed fields. The proof of Proposition~\ref{ai} goes through;
using  Lemma~\ref{dto} and the definable total ordering
of $K$ it gives the following for sets $S\subseteq K^n$ that are definable in 
 $(K,\k)$: 

\begin{lemma} If $\dim_2 S=0$, then $S$ is internal to $\k$.
\end{lemma}

\medskip\noindent
We now equip $K$ with its order topology and each $K^n$ with
the corresponding product topology.

\begin{lemma} Suppose $S\subseteq K$ is definable in $(K,\k)$.  Then
$S$ is either discrete, or has nonempty interior in $K$. If $S$ is discrete, 
then $\dim_2 S \le 0$;
if $S$ has nonempty interior, then $\dim_2 S=1$.
\end{lemma}
\begin{proof} The first assertion is a consequence of \cite[7.4]{Dol},
which states the local o-minimality of
$(K,\k)$. Let $S$ be discrete. Passing to a suitable elementary extension of
$(K,\k)$ we can assume that $(K,\k)$ is $\omega$-saturated. Then the proof (or proof sketch) of \cite[7.4]{Dol} shows that  $S\subseteq \text{dcl}(\k(a))$ for some
$a\in K^n$, where
the definable closure is taken in the real closed field $K$. This definable 
closure is the relative algebraic closure of the field $\k(a)$ in $K$, so
for every $b\in S$ there exists a $P(X,Y,Z)\in \Z[X,Y,Z]$ as before
such that $P(u,a,b)=0$ and $P(u,a,Z)\ne 0$ for some $u\in \k^m$. By saturation only finitely many 
$P$'s are needed here, so $S$ is contained in a finite union of $\Phi_2$-sets,
and thus $\dim_2 S \le 0$. 

It is routine to check that if $S$ has nonempty interior in $K$, 
then $\dim_2 S \ne 0$, hence $\dim_2 S =1$.  
\end{proof} 

\noindent
Thus $\dim_2 S=0 \text{ if and only if } S \text{ is discrete}$,
where we assume that $S\subseteq K$ is definable in $(K,\k)$. Since 
{\it discrete\/} is a first-order condition, we conclude:

\begin{cor} $(\Phi_2, \RCF2)$ is bounded.
\end{cor}

\begin{prop} Let $S\subseteq K^n$ be nonempty and definable. Then
$$\dim_2 S=0 \text{ if and only if } S \text{ is discrete}.$$
\end{prop}
\begin{proof} Let $\pi_i:K^n\rightarrow K$ be given by $\pi_i(x_1,\dots,x_n)=x_i$, for $i\in \{1,\dots,n\}$. Suppose $\dim S=0$. Then $\dim_2 \pi_i(S)=0$ by lemma \ref{dim3}, so $\pi_i(S)$ is discrete for each $i$, by the above. Thus the product $\pi_1(S)\times \dots \times\pi_n(S)$ is discrete, and so is its subset $S$.

\medskip\noindent
For the converse we use the same argument as in the proof of \cite[Proposition 4.1]{D3}. Assume $S\subseteq K^n$ is discrete. We first replace 
$(K,\k)$ by a suitable countable elementary substructure over which $S$ is defined and $S$ by its corresponding trace.
Now that $K$ is countable we next pass to its completion $K^{\operatorname{c}}$ as defined in~\cite[Section~4.4]{D4}, which gives 
an elementary extension $(K^{\operatorname{c}},\k)$ of $(K,\k)$. Replacing $K$ by $K^{\operatorname{c}}$
and $S$ by the corresponding
extension, the overall effect is that we have arranged
$K$ to be {\em uncountable,}\/ but with a {\em countable}\/ base for its topology. 
Then the discrete set $S$ is countable, so $\pi_i(S)\subseteq K$
is countable for each $i$, hence with empty interior, so
$\dim \pi_i(S)=0$ for all $i$, and thus
$\dim S =0$. 
\end{proof}

\noindent
These results and earlier generalities yield Theorem~\ref{mtr}


\textsc{Departamento de Matem\'aticas, Universidad de los Andes, Cra. 1. No. 18A-10, Bogotá, Colombia.}

\textit{E-mail address:} jl.angel76@uniandes.edu.co

\medskip\noindent
\textsc{Department of Mathematics, University of Illinois at Urbana-Champaign, 1409 W. Green Street, Urbana, IL 61801.}

\textit{E-mail address:} vddries@math.uiuc.edu
\end{document}